\documentclass[]{amsart}

\usepackage{amsfonts}
\usepackage{amscd}
\usepackage{epsfig}
\usepackage{epic,eepic}
\usepackage{graphicx}
\usepackage{psfrag}
\usepackage{amsmath,mathrsfs,amssymb}
\usepackage{amsthm}
\usepackage{setspace}
\usepackage{url}
\usepackage{here}
\usepackage{diagrams}

\newtheorem{theorem}{Theorem}[section]
\newtheorem{lemma}[theorem]{Lemma}
\newtheorem{proposition}[theorem]{Proposition}
\newtheorem{corollary}[theorem]{Corollary}
\newtheorem{question}[theorem]{Question}

\theoremstyle{definition}
\newtheorem{remark}[theorem]{Remark}
\newtheorem{definition}[theorem]{Definition}

\newcommand{\Pth}{\mathbb P^{3}}
\newcommand{\Pt}{\mathbb P^{2}}
\newcommand{\BD}{{\bf B}($D$)}
\newcommand{\Pn}{\mathbb P^{n}}
\newcommand{\OO}{\mathcal O}
\newcommand{\Gn}{\mathbb G(n-2, n)}
\newcommand{\Go}{\mathbb G(1,n)}

\title{Hilbert scheme of a pair of codimension two linear subspaces}

\author{Dawei Chen}

\address{University of Illinois at
  Chicago, Department of Mathematics, Statistics and Computer Science, Chicago, IL 60607}

\email{dwchen@math.uic.edu}

\author{Izzet Coskun}

\address{University of Illinois at
  Chicago, Department of Mathematics, Statistics and Computer Science, Chicago, IL 60607}

\email{coskun@math.uic.edu}

\author{Scott Nollet}

\address{Texas Christian University, Department of Mathematics, 
Fort Worth, TX 76129}

\email{s.nollet@tcu.edu}

\subjclass[2000]{Primary: 14E05, 14E30, 14M15, 14D22}

\thanks{During the preparation of this article the first author was
  partially supported by the 2009 Spring Algebraic Geometry Program at MSRI and the second author was partially 
supported by the NSF grant DMS-0737581 and an Alfred P. Sloan Foundation Fellowship.}

\begin{document}
\bibliographystyle{plain}

\begin{abstract}
We study the component $H_n$ of the Hilbert scheme whose general point parameterizes a 
pair of codimension two linear subspaces in $\Pn$ for $n\geq 3$. 
We show that $H_n$ is smooth and isomorphic to the blow-up of the symmetric square of 
$\Gn$ along the diagonal. 
Further $H_n$ intersects only one other component in the full Hilbert scheme, transversely. 
We determine the stable 
base locus decomposition of its effective cone and give modular interpretations of the 
corresponding models, hence conclude that $H_n$ is a Mori dream space. 
\end{abstract}

\maketitle

\section{Introduction}
The Hilbert scheme Hilb$^{p(m)} (\Pn)$ parameterizes closed subschemes in $\Pn$ with fixed Hilbert polynomial $p(m)$. 
Grothendieck \cite{G} proved that Hilb$^{p(m)} (\Pn)$ exists as a projective scheme and Hartshorne \cite{Ha} 
showed that it is connected. In general, Hilb$^{p(m)} (\Pn)$ can be very complicated, possibly having many components of 
various dimensions \cite{Ch1,Ch2} or generically non-reduced components \cite{MDP}.   
Investigating the geometry of particular components has been one of the main themes in the study of the Hilbert schemes. 
For example, Piene and Schlessinger showed that Hilb$^{3m+1} (\Bbb P^3)$ has two smooth components which meet transversely 
and gave an explicit description of the component whose general member is a twisted cubic curve \cite{PS}. 

In this paper, we study the component of the Hilbert scheme whose general point parameterizes a pair of codimension two linear 
subspaces in $\Pn$. Let $X$ be a pair of general codimension two linear subspaces $\Lambda_{n-2}$ and $\Lambda'_{n-2}$ in $\Pn$ 
that intersect along a codimension four linear subspace $\Lambda_{n-4}$. The exact sequence 
$$0\rightarrow \OO_{X}\rightarrow \OO_{\Lambda_{n-2}}\oplus \OO_{\Lambda'_{n-2}}\rightarrow \OO_{\Lambda_{n-4}}\rightarrow 0$$
implies that $X$ has the Hilbert polynomial  
$$P_n(m) = 2{n-2+m\choose m } - {n-4+m\choose m}.$$ 
Since a degree two irreducible, reduced, codimension two subscheme of $\Pn$ is contained in a hyperplane but $X$ is not, 
there exists an irreducible component $H_n$ of the Hilbert scheme Hilb$^{P_n}(\Pn)$ whose general point parameterizes $X$. 

For $n=2$, it is well-known that $H_2$ is the full Hilbert scheme Hilb$^2 (\Pt)$ parameterizing length-2 zero dimensional 
subschemes of $\Pt$ and is isomorphic to the blow-up of Sym$^2 \Pt$ along the diagonal. For $n=3$, the Hilbert polynomial 
of a pair of skew lines in $\Pth$ is $2m+2$. The structure of Hilb$^{2m+2}(\Pth)$ was sketched in \cite[1.b]{H} and 
elaborated in \cite[3.5, 4.2]{L}. It consists of two irreducible components $H_3$ and $H'_3$, of respective dimensions 8 and 11.
The general point of $H_3$ parameterizes a pair of skew lines while the general point of $H'_3$ parameterizes a plane conic 
union an isolated point. 
%In \cite[1.b]{H} and \cite[Conjecture 3.5.10]{L} it was suggested that both $H_3$ and $H'_3$ are smooth, 
%but a deformation calculation is needed to verify the smoothness. For $n\geq 4$, the number of irreducible components of 
%Hilb$^{P_n}(\Pn)$ grows rapidly with $n$, cf. Remark~\ref{number}. Nevertheless, our results below provide a good understanding 
%of the component $H_n$. 
In Theorem \ref{1} below, we prove that both $H_3$ and $H'_3$ are smooth, as suggested in \cite[1.b]{H} and \cite[Conjecture 3.5.10]{L}. 
Moreover, in spite of the rapid growth in the number of irreducible components of Hilb$^{P_n (m)}(\Pn)$ for $n > 3$ (Remark \ref{number}), 
we provide a good understanding of the component $H_n$ for all $n$, showing it to be smooth, isomorphic to the blow-up 
of Sym$^2 \Gn$ along the diagonal, and completely working out its Mori theory.  

\begin{theorem}
\label{1} 
Let $n\geq 3$ be an integer. 

(1) A subscheme parameterized by $H_n$ is projectively equivalent to one of the following four types:

 (I) A pair of codimension two linear subspaces intersecting along a codimension four linear subspace. 
 
 (II) A pure double structure supported on a codimension two linear subspace.
 
(III) A pair of codimension two linear subspaces intersecting along a codimension three linear subspace with an embedded component determined by the square of the ideal of the intersection. 

(IV) A double structure contained in a hyperplane and supported on a codimension two linear subspace with an embedded component determined by the square of the ideal of 
a codimension three linear subspace. 

The loci (I), (II), (III) and (IV) have dimensions $4n-4$, $4n-5$, $3n-2$ and $3n-3$, respectively. The closure of (I) is $H_n$. The closures of (II) and (III) intersect along (IV). 

(2) In the full Hilbert scheme Hilb$^{P_n}(\Pn)$, $H_n$ intersects only one other component $H'_n$ of dimension $7n-10$ whose general point parameterizes a quadric $(n-2)$-fold $Q$ union a codimension three linear subspace $\Lambda_{n-3}$, where $Q\cap \Lambda_{n-3}$ is a codimension four linear subspace. Moreover, $H_n$ and $H'_n$ intersect transversely along the loci (III) $\cup$ (IV).  

(3) The component $H_n$ is smooth and isomorphic to the blow-up of Sym$^{2}\Gn$ along the diagonal. For $n=3$, the other component $H'_3$ is smooth and isomorphic to the blow-up of $\Pth\times \mbox{Hilb}^{2m+1}(\Pth)$ along the incidence correspondence $\{p\in C\}$, where $p$ denotes a point in $\Pth$ and $C$ denotes a conic parameterized by Hilb$^{2m+1}(\Pth)$.  
\end{theorem}

To study the Mori theory of $H_n$, we introduce the following divisor classes on $H_n$.

\begin{definition}
\label{divisors}
Let $n\geq 3$ be an integer.  
 
Let $M$ be the divisor class of the locus of subschemes that intersect a fixed line. 

Let $N$ be the divisor class of the locus of generically non-reduced subschemes. 

Consider the locus of subschemes whose intersection with a fixed plane consists of two points, which are collinear with a fixed point on that plane. 
Let $F$ be the divisor class parameterizing the closure of this locus in $H_n$. 

Let $E$ be the divisor class of the locus of subschemes such that the intersection of the two subspaces in the pair intersects a fixed $\Pth$.  
For $n=3$, $E$ parameterizes the locus of two incident lines with a spatial embedded point at their intersection.  
\end{definition}

Since $H_n$ is smooth, the Weil divisors defined above are Cartier. For a divisor $D$, let $\BD$ be its stable base locus. Denote by $[D_{1}, D_{2}]$, $(D_{1}, D_{2})$ and $[D_{1}, D_{2})$ the convex cones consisting of divisors of type $aD_{1}+bD_{2}$, where $a,b\geq0$, $a,b>0$ and $a>0, b\geq 0$, respectively. Our next result describes the stable base locus decomposition for the effective cone of $H_n$. 

\begin{theorem}
\label{2} 
Let $n\geq 3$ be an integer. 

(1) The Picard group of $H_n$ is generated by $M$ and $F$. The divisor class $N$ is linearly equivalent to $2M-2F$. The divisor 
class $E$ is linearly equivalent to $2F-M$. Moreover, two divisors on $H_n$ are linearly equivalent iff they are numerically equivalent. 

(2) The ample cone of $H_n$ is $(F,M)$. The effective cone of $H$ is $[N, E]$. For a divisor $D$ in the  
chamber $[F, M]$, $D$ is base-point-free. For $D$ in the chamber $(M, N]$, $\BD$ consists of the loci (II) and (IV). For $D$ in the chamber 
$[E, F)$, $\BD$ consists of the loci (III) and (IV). 
\end{theorem}

\begin{definition}
For an effective divisor $D$ on a variety $X$, let $P(D)$ denote its Proj model 
$$\mbox{Proj}\big(\bigoplus\limits_{m\geq 0}H^{0}(X, mD)\big) $$
assuming the section ring of $D$ is finitely generated. Let $\psi_D$: $X\dashrightarrow P(D)$ 
(morphism or rational map) denote the map induced by $D$. 
\end{definition}

In order to describe all possible models of $H_n$, we define two spaces $\Psi_n$ and $\Theta_n$ as follows. 

\begin{definition}
\label{spaces}
Let $\Psi_n$ denote the $\mathbb G(3,5)$ bundle over $\mathbb G(3,n)$ whose fiber over a base point $[\Lambda_3]$ parameterizes codimension two linear sections 
of the Pl\"{u}cker embedding of the Grassmannian of lines in $\Lambda_3$. In particular, $\Psi_3$ is isomorphic to $\mathbb G(3,5)$. 

Consider the Pl\"{u}cker embedding $\Gn\cong\Go\hookrightarrow \mathbb P^N$. There is a subset of $\mathbb G(N-2, N)$ parameterizing 
codimension two linear sections of $\Gn\cong\Go$ that are the intersections of two Schubert varieties $\Sigma_1\cap \Sigma'_1$. Let $\Theta_n$ denote (the normalization of) the closure of 
this subset in $\mathbb G(N-2, N)$. 
\end{definition}

The reader can refer to Remark~\ref{quadrics} for another geometric interpretation of $\Psi_n$ and $\Theta_n$. 

Our third result describes the model $P(D)$ for any effective divisor $D$ on $H_n$. 
Since the locus (III) is divisorial iff $n=3$, the results for $n=3$ and $n\geq 4$ are slightly different. 

\begin{theorem} 
\label{3}
Let $n\geq 3$ be an integer. Let $D$ denote an effective divisor on $H_n$. 

(1) For $D$ in the chamber $(F, M)$, the model $P(D)$ is isomorphic to $H_n$. 

(2) For $D$ in the chamber $[M, N)$, the morphism $\psi_D$ contracts the loci (II) and (IV). 
The resulting model $P(D)$ is isomorphic to Sym$^{2} \Gn$. 

(3) The morphism $\psi_F$ contracts the loci (III) and (IV). The model $P(F)$ is isomorphic to $\Theta_n$. 

(4) For $n=3$ and $D$ in the chamber $(E,F)$, the morphism $\psi_D$ contracts the divisor $E$ and $P(D)$ is isomorphic to $\Psi_3\cong\mathbb G(3,5)$. 

(5) For $n\geq 4$ and $D$ in the chamber $(E,F)$, the birational map $\psi_D$ is a flip over $\Theta_n$ and the flipping space $P(D)$ is isomorphic to $\Psi_n$. Further the birational transform of $E$ on $\Psi_n$ induces a morphism that contracts the $\mathbb G(3,5)$ bundle structure to the base $\mathbb G(3, n)$.  
\end{theorem}

A variety $X$ is called a Mori dream space if Mori's program can be carried out for every effective divisor on $X$ \cite{HK}. 
By Theorems~\ref{2} and \ref{3}, this holds for $H_n$.

\begin{corollary}
\label{Mori}
The Hilbert component $H_n$ is a Mori dream space.
\end{corollary}

This corollary also follows from \cite[Corollary 1.3.1]{BCHM} in view of Propositions~\ref{KH} and \ref{KW}. 

\begin{question}{\em
More generally, one may consider the component $H(a,b,n)$ of the Hilbert scheme whose general member is a union of 
two linear subspaces in $\Pn$ of codimensions $a$ and $b$, intersecting in the expected codimension. For which triples
$(a,b,n)$ is this component smooth or a Mori dream space? In the case $a=b$, it would be interesting to determine the 
birational relation between $H(a,a,n)$ and Sym$^2 \mathbb G(n-a,n)$. 
\em}\end{question}

This paper is organized as follows. In section 2, we describe the stratification of $H_n$ and prove Theorem~\ref{1}. In section 3, we study 
the divisor theory of $H_n$ and prove Theorems~\ref{2}, \ref{3}. Throughout the paper, we work over an algebraically closed field $k$ of characteristic zero.
We often denote an $m$-dimensional linear subspace of $\Pn$ by $\Lambda_{m}$. The ideal of a subscheme of $\Pn$ is always saturated. All divisors considered here are Cartier. 

{\bf Acknowledgements.} We would like to thank Lawrence Ein, Tommaso de Fernex, Joe Harris, Anatoly Libgober, Mihnea Popa and Christian Schnell for useful conversations related to this paper. Part of this work was done when the first two authors were visiting MSRI for  the  Spring 2009 Algebraic Geometry Program. They would like to thank MSRI for providing support and a wonderful, stimulating work environment.

\section{Description of $H_n$}
In this section, let $n\geq 3$ denote an integer. Let $S=k[x_{0},\ldots, x_{n}]$ denote the coordinate ring of $\Pn$. 
We begin by determining the double structures of pure dimension supported on a codimension two linear subspace of $\Pn$. 
The following lemma generalizes the classification of double lines given in $\Bbb P^3$ \cite{M,N}. 

\begin{lemma}
\label{double}
Let $X$ be a pure codimension two subscheme of $\Pn$ which is a double structure supported on a codimension two linear subspace 
$\Lambda_{n-2}$: $x_{0} = x_{1} = 0$. Then the ideal of $X$ can be written as $(x_{0}^{2}, x_{0}x_{1}, x_{1}^{2}, x_{0}G - x_{1}F)$, 
where $F$ and $G$ are degree $k$ homogeneous polynomials in $x_{2},\ldots, x_{n}$ without common factors. 
Moreover, the Hilbert polynomial of $X$ equals $P_n$ iff $F$ and $G$ are linear. 
\end{lemma}

\begin{proof}
The ideal $I_{X}$ of $X$ satisfies $I_{\Lambda_{n-2}}^{2}\subset I_{X}\subset I_{\Lambda_{n-2}}$. Therefore, $I_{X}$ contains the ideal 
$(x_{0}^{2}, x_{0}x_{1}, x_{1}^{2})$ and a polynomial of type $x_{0}G_0 - x_{1}F_0$, where $F_0$ and $G_0$ are homogenous polynomials of the same degree in $x_{2},\ldots, x_{n}$. Suppose $G_0=G H$ and $F_0=F H$, where $F, G$ are relatively prime of degree $k$. 
Then the ideal $((x_0, x_1)^{2}, x_0 G - x_1 F)$ is the total ideal for a double structure $Y$ on $\Lambda_{n-2}$ without embedded components.
Away from the proper closed subset $H=0$, we have the containment 
$X \subset Y$, and therefore $X = Y$ since $X$ has no embedded components. Moreover, $I_{X}=I_{Y}=((x_0, x_1)^{2}, x_0 G - x_1 F)$. 

Let us compute the Hilbert polynomial of $I_{X}$. For $m\gg 0$, an element $H\in (S/I_{X})_{m}$ can be written as 
$A +x_{0}B_{0} - x_{1}B_{1}$, where $A$ and $B_{i}$ are homogeneous polynomials of degree $m$ and $m-1$, respectively, 
in $x_{2},\ldots, x_{n}$. Moreover, $x_{0}B_{0} - x_{1}B_{1}$ is divisible by $x_{0}G - x_{1}F$ iff $B_{0} = CG$ and $B_{1} = CF $ for a 
degree $m-1-k$ homogeneous polynomial $C$ in $x_{2},\ldots, x_{n}$. Hence, we know dim $(S/I_{X})_{m} = {n-2+m\choose m} + 2{n-2+m-1\choose m-1} - {n-2+m-1-k\choose m-1-k}$, which equals $P_n$ iff $k=1$. 
\end{proof}

Now we classify all subschemes parameterized by $H_n$ up to projective equivalence. 

\begin{proof}[Proof of Theorem~\ref{1} (1)]
We want to show any subscheme $X$ parameterized by $H_n$ belongs to one of the four loci. Let $I_X$ denote the ideal of $X$ and $X_{red}$ denote the purely $(n-2)$-dimensional reduced part of $X$. Note that $X_{red}$ has degree two or one. In the former case, $X_{red}$ consists of a pair of codimension two linear subspaces. If their intersection has dimension $n-4$, then the Hilbert polynomial of $X_{red}$ equals $P_n$, hence $X = X_{red}$. Without loss of generality, assume $I_X$ is $(x_{0}x_{2}, x_{0}x_{3}, x_{1}x_{2}, x_{1}x_{3})$, namely, $X$ consists of two linear subspaces $x_{0} = x_{1} = 0$ and $x_{2}=x_{3}=0$. This corresponds to the locus (I). 

If the two components of $X_{red}$ intersect along a codimension three linear subspace, without loss of generality, assume 
$I_{X_{red}}$ is $(x_{0}, x_{1}x_{2})$, namely, it consists of two linear subspaces $x_{0}=x_{1}=0$ and $x_{0} = x_{2} =0$. 
Note that $I_X$ is contained in $I_{X_{red}}$. Since a one dimensional flat family in (I) specializes to $X$, there is a spatial embedded component 
of $X$ whose support is contained in the linear intersection of the two components of $X_{red}$. Then $I_X$ is contained in $(x_{0}, x_{1}x_{2})\cap (x_{0}, x_{1}, x_{2})^{2} = (x_{0}^2, x_{0}x_{1}, x_{0}x_{2}, x_{1}x_{2})$, whose Hilbert polynomial equals $P_n$. Hence, $I_X$ 
equals $(x_{0}^{2}, x_{0}x_{1}, x_{0}x_{2}, x_{1}x_{2})$ and $X$ has an embedded 
component supported on the codimension three linear intersection of the two components. This corresponds to the locus (III).
The embedded structure is uniquely determined by the square of the ideal of the codimension three linear intersection. 
One can take a family with general member  $(x_{0}, x_{1})\cap (x_{0} + tx_3, x_{2})$ whose flat limit is $X$. Hence, 
the locus of (III) is in the closure of the locus (I). 

If $X_{red}$ has degree one, it is a codimension two linear subspace. Hence, $X$ is a generically double structure supported on $X_{red}$. 
Suppose $X_{red}$ is defined by $x_{0} = x_{1}= 0$. Let $X'$ be the non-reduced subscheme of $X$ of pure dimension $n-2$ supported on $X_{red}$. 
By Lemma~\ref{double}, $I_{X'}$ equals $(x_{0}^{2}, x_{0}x_{1}, x_{1}^{2}, x_{0}G - x_{1}F)$, where 
$F$ and $G$ are degree $k$ homogeneous polynomials in $x_{2},\ldots, x_{n}$ without common factors. Since $h^{0}(I_X(2))$ has dimension $\geq 4$ by
semi-continuity and $I_X$ is contained in $I_{X'}$, we know $F, G$ must be linear and $I_X$ equals $(x_{0}^{2}, x_{0}x_{1}, x_{1}^{2}, x_{0}G - x_{1}F)$, 
whose Hilbert polynomial is $P_n$ by Lemma~\ref{double}. This corresponds to the locus (II). 
One can take a flat family with general member $(x_{0}, x_{1})\cap (x_{0} + tF, x_{1}+ tG)$ that specializes to $X$, where $F$ and $G$ are defined as above.  
Hence, the locus of (II) is in the closure of the locus (I). 

If $F$ and $G$ are linearly dependent, one can assume $F = G = x_{2}$. Since $I_X$ contains $(x_{0}^{2}, x_{0}x_{1}, x_{1}^{2}, x_{0}x_{2}-x_{1}x_{2})$, whose Hilbert polynomial equals $P_n$, $I_X$ must equal $(x_{0}^{2}, x_{0}x_{1}, x_{1}^{2}, x_{0}x_{2}-x_{1}x_{2}) = (x_{0} -  x_{1}, x_{0}^{2}) \cap (x_{0}, x_{1}, x_{2})^{2}$. 
From this expression, we see that $X$ consists of a double structure contained in the hyperplane $x_{0} - x_{1} = 0$ along with an embedded 
component supported on the codimension three linear subspace $x_{0} = x_{1} = x_{2} = 0$. This corresponds to the locus (IV). 

The dimension counts for the above loci are standard. The locus (I) is open and dense in $H_n$. We have seen that subschemes of type (II) can degenerate to (IV). For an ideal $(x_{0}^{2}, x_{0}x_{1}, x_{0}x_{2}, x_{1}x_{2})$ of type (III), one can replace $x_2$ by $x_1 + tx_2$. The flat limit lies in (IV). So subschemes of type (III) can also degenerate to (IV). 
\end{proof}

Let $\Delta$ denote the diagonal of Sym$^{2} \Gn$. One can regard Sym$^{2} \Gn$ as the Chow variety parameterizing a pair of codimension two and degree one cycles in $\Pn$. 
Let Bl$_{\Delta}$Sym$^{2} \Gn$ denote the blow-up of Sym$^{2} \Gn$ along the diagonal. 

\begin{lemma}
Bl$_{\Delta}$Sym$^{2} \Gn$ is a smooth variety. 
\end{lemma}

\begin{proof}
Let $X$ be a nonsingular variety. Denote by $Y$ the blow-up of $X\times X$ along the diagonal. There is a natural involution acting on $Y$ and the quotient space is $\mbox{Bl}_{\Delta}\mbox{Sym}^{2}X$. Since the smooth exceptional divisor is the fixed locus, we conclude that $\mbox{Bl}_{\Delta}\mbox{Sym}^{2}X$ is smooth. In particular, $\mbox{Bl}_{\Delta}\mbox{Sym}^{2}\Gn$ is smooth. 
\end{proof}

\begin{proposition}
\label{bijective}
There is a bijective morphism $\delta$: Bl$_{\Delta}\mbox{Sym}^{2} \Gn \rightarrow H_n$. 
\end{proposition}

\begin{proof}
A generically reduced subscheme parameterized by $H_n$ is uniquely 
determined by a pair of codimension two linear subspaces, cf. Theorem~\ref{1} (1). 
Hence, there is a natural bijection between Sym$^{2}\Gn\backslash \Delta$ and $H_n\backslash\mbox{(II)}\cup\mbox{(IV)}$. Fix a codimension two linear subspace $\Lambda_{n-2}$ given by $x_{0} = x_{1} = 0$ and consider another linear subspace $x_{0}+tF = x_{1}+tG = 0$ approaching $\Lambda_{n-2}$ as $t\rightarrow 0$, where $F$ and $G$ are linear functions in $x_{2},\ldots, x_{n}$. Note that $(F, G)$ can be regarded as an element $\phi\in $ Hom $(\mathbb A^{n-1}, \mathbb A^{n+1}/\mathbb A^{n-1})$ of the tangent space of $\Gn$ at $[\Lambda]$. We have seen in the proof of Theorem~\ref{1} (1) that if $F$ and $G$ are linearly independent, the limit scheme is uniquely determined in (II). If $F$ and $G$ are dependent, then the limit scheme is determined in (IV). Using the universal property of the Hilbert scheme, we thus obtain the desired bijective morphism.  
\end{proof}

Next, we will prove Theorem~\ref{1} (2) regarding the smoothness of $H_n$ and how it intersects other components in the full Hilbert scheme. 
Note that if a point parameterizing a subscheme $X$ is in the singular locus of $H_n$, then all the points parameterizing subschemes projectively equivalent 
to $X$ lie in the singular locus of $H_n$. Since the locus (IV) is contained in the closure of the locus (I), (II) or (III) and each locus is homogeneous, if $H_n$ is smooth along (IV), then it must be smooth everywhere. So it suffices to analyze the deformation space of a subscheme of type (IV). We invoke the following result, which transforms the study of the deformation of a subscheme to that of its ideal. 

\begin{theorem}[Comparison Theorem, \cite{PS}]
\label{comparison}
If the ideal $I$ defining a subscheme $X\subset \Pn$ is generated by homogeneous polynomials $f_{1},\ldots, f_{r}$ of degrees $d_{1},\ldots, d_{r}$, for which $$(k[x_{0},\ldots, x_{n}]/I)_{d}\cong H^{0}(\OO_{X}(d)) $$
for $d = d_{1},\ldots, d_{r}$, then there is an isomorphism between the universal deformation space of $I$ and that of $X$. 
\end{theorem}

From now on, fix a subscheme $X$ of type (IV) with ideal $I = (x_{0}^{2}, x_{0}x_{1}, x_{1}^{2}, x_{0}x_{2})$. 

\begin{lemma}
\label{localcohomology}
The hypothesis of the Comparison Theorem holds for X.
\end{lemma}

\begin{proof}
Let $J = (x_{0}, x_{1}^{2})$ be the double structure contained in $x_{0} = 0$ and supported on $x_{0}=x_{1} = 0$. Let $K = J/I$, which is isomorphic to $S/(x_{0}, x_{1}, x_{2})$ twisted by $-1$ as an $S$-module. Using the exact sequence 
$$ 0 \rightarrow H_{\bf m}^{0}(M)\rightarrow M\rightarrow \bigoplus\limits_d H^{0}(\widetilde M(d))\rightarrow H_{\bf m}^1(M) \rightarrow 0 $$
where $M$ is a graded $S$-module and $\widetilde M$ is the corresponding quasicoherent sheaf on $\Pn$, 
we get the positive graded pieces of the local cohomology $H_{\bf m}^{i} (K)$ and $H_{\bf m}^{i} (S/J)$ are vanishing for $i=0, 1$. 
Then the local cohomology sequence associated to the exact sequence $$0\rightarrow K\rightarrow S/I\rightarrow S/J\rightarrow 0$$
shows that $(S/I)_{d}\rightarrow H^{0}(\mathcal O_{X}(d))$ is an isomorphism for all $d > 0$. Therefore, the completion of Hilb$^{P_n}(\Pn)$ at $[X]$ 
can be identified as the universal deformation space of the ideal $I$ of $X$. 
\end{proof}

We follow the method in \cite{PS} to write down a universal deformation space of the ideal $I$. 

\begin{proposition}
\label{deformation}
The tangent space of Hilb$^{P_n}(\Pn)$ at $[X]$ has dimension $8n-12$. The ideal $I$ of $X$ has a universal deformation space of type
$\mathbb A^{4n-4}\cup \mathbb A^{7n-10}$, where the two components intersect transversely along $\mathbb A^{3n-2}$. 
\end{proposition}

\begin{proof}
By Lemma~\ref{localcohomology}, the tangent space of Hilb$^{P_n}(\Pn)$ at $[X]$ can be identified as Hom$_{S}(I/I^{2}, S/I)_{0}$. Consider the following
presentation of $S/I$ over $S = k[x_{0},\ldots, x_{n}]$: 
$$\begin{CD}
0 @>>> S(-4)@>\nu>> S(-3)^{4}@>\mu>> S(-2)^{4}@>\lambda>> S@>>> S/I @>>> 0, 
\end{CD}$$
where the maps are given by 
$$\lambda = (x_{0}x_{1},x_{0}x_{2},x_{0}^{2},x_{1}^{2}),\  
\mu = \left( \begin{array}{cccc} 
x_{1} & x_{2} & x_{0} & 0 \\
0 & -x_{1} & 0 & x_{0} \\
0 & 0 & -x_{1} & -x_{2} \\
-x_{0} & 0 & 0 & 0 \end{array} \right),\  
\nu = \left(\begin{array}{c}
0 \\
x_{0} \\
-x_{2} \\
x_{1} \end{array} \right) $$
An element $\phi\in$ Hom$_{S}(I/I^{2}, S/I)_{0}$ satisfies 
$$x_{1}\phi (x_{0}x_{1}) = x_{0}\phi (x_{1}^{2}), \ x_{2}\phi(x_{0}x_{1}) = x_{1}\phi(x_{0}x_{2}), $$ 
$$x_{0}\phi (x_{0}x_{1}) = x_{1}\phi (x_{0}^{2}), \ x_{0}\phi(x_{0}x_{2}) = x_{2}\phi(x_{0}^{2})$$
modulo $I$. Then one can check that Hom$_{S}(I/I^{2}, S/I)_{0}$ is generated by the following elements: 
$$\phi(x_{0}x_{1}) = x_{0}\sum_{i\geq 3}a_{i}x_{i}+x_{1}\sum_{i\geq 2}b_{i}x_{i}, $$
$$\phi(x_{0}x_{2}) = x_{0}\sum_{i\geq 3}c_{i}x_{i}+x_{1}\sum_{i\geq 2}d_{i}x_{i}+x_{2}\sum_{i\geq 2}b_{i}x_{i}, $$
$$\phi(x_{0}^{2}) = x_{0}\sum_{i\geq 3}e_{i}x_{i}, $$
$$\phi(x_{1}^{2}) = x_{0}\sum_{i\geq 3}f_{i}x_{i}+x_{1}\sum_{i\geq 2}g_{i}x_{i}+x_{2}\sum_{i\geq 2}h_{i}x_{i}, $$
where $a_{i}, b_{i}, \ldots, h_{i}$ are independent parameters. Hence, Hom$_{S}(I/I^{2}, S/I)_{0}$ has dimension $8n-12$. 

Let us write down a group of generators for Hom$_{S}(I/I^{2}, S/I)_{0}$. For $3\leq i \leq n$, let 
$$ \frac{\partial}{\partial t_{0i}} = x_{i}\frac{\partial}{\partial x_{0}}=\left( \begin{array}{c} 
x_{1}x_{i} \\
x_{2}x_{i} \\
2x_{0}x_{i} \\
0 \end{array}\right), \ 
\frac{\partial}{\partial t_{1i}} = x_{i}\frac{\partial}{\partial x_{1}}=\left( \begin{array}{c} 
x_{0}x_{i} \\
0 \\
0 \\
2x_{1}x_{i} \end{array}\right), \ 
\frac{\partial}{\partial t_{2i}} = x_{i}\frac{\partial}{\partial x_{2}}=\left( \begin{array}{c} 
0 \\
x_{0}x_{i} \\
0 \\
0 \end{array}\right) $$ 
Also let
$$ \frac{\partial}{\partial t_{01}} = x_{1}\frac{\partial}{\partial x_{0}}=\left( \begin{array}{c} 
0 \\
x_{1}x_{2} \\
0 \\
0 \end{array}\right), \ 
\frac{\partial}{\partial t_{02}} = x_{2}\frac{\partial}{\partial x_{0}}=\left( \begin{array}{c} 
x_{1}x_{2} \\
x_{2}^{2} \\
0 \\
0 \end{array}\right), \ 
\frac{\partial}{\partial t_{12}} = x_{2}\frac{\partial}{\partial x_{1}}=\left( \begin{array}{c} 
0 \\
0 \\
0 \\
2x_{1}x_{2} \end{array}\right) $$
Note that $X$ uniquely determines a $(\Lambda_{n-3}\subset\Lambda_{n-2}\subset\Lambda_{n-1})$ flag and vice versa, where $\Lambda_{k}$ is a $k$-dimensional linear subspace of $\Pn$. The above $3n-3$ elements $\frac{\partial}{\partial t_{ij}}, 0\leq i\leq 2, i<j\leq n$ provide the trivial deformations 
for the ideal $I$ of $X$, which correspond to moving the flag determined by $X$. 

Moreover, for $i\geq 3$, let 
$$ \frac{\partial}{\partial u_{1i}} = \left( \begin{array}{c} 
 0 \\
x_{1}x_{i} \\
0 \\
0 \end{array}\right), \ 
\frac{\partial}{\partial u_{2i}} = \left( \begin{array}{c} 
 0 \\
0 \\
x_{0}x_{i} \\
0 \end{array}\right), \  
\frac{\partial}{\partial u_{3i}} = \left( \begin{array}{c} 
 0 \\
0 \\
0 \\
x_{1}x_{i} \end{array}\right), $$ 
$$ \frac{\partial}{\partial u_{4i}} = \left( \begin{array}{c} 
 0 \\
0 \\
0 \\
x_{2}x_{i} \end{array}\right), \ 
\frac{\partial}{\partial u_{5i}} = \left( \begin{array}{c} 
 0 \\
0 \\
0 \\
x_{0}x_{i} \end{array}\right), \ 
\frac{\partial}{\partial u_{6}} = \left( \begin{array}{c} 
 0 \\
0 \\
0 \\
x_{2}^{2} \end{array}\right)$$
These $\frac{\partial}{\partial u_{ij}}, 1\leq i\leq 5, 3 \leq j\leq n$ and $\frac{\partial}{\partial u_{6}}$ provide a versal deformation space for $I$. Along with those
 $\frac{\partial}{\partial t_{ij}}$, they form a basis for Hom$_{S}(I/I^{2}, S/I)_{0}$. For $1\leq i\leq 5$, define 
 $$ v_{i} = \sum_{j\geq 3}u_{ij}x_{j}.$$
 Consider the following homogeneous perturbations of $\lambda, \mu$ and $\nu$: 
$$ \lambda' = (x_{0}x_{1}-u_{6}x_{2}v_{1}, \ x_{0}x_{2}+x_{1}v_{1}, \ x_{0}^{2}+x_{0}v_{2}+u_{6}v_{1}^{2}, $$
$$x_{1}^{2} + x_{1}v_{3} + x_{2}v_{4}+x_{0}v_{5} + u_{6}x_{2}^{2} + v_{2}v_{5}), $$

$$ \mu' = \left(\begin{array}{cccc}
x_{1}+v_{3} & x_{2} & x_{0}+v_{2} & -v_{1} \\
v_{4}+u_{6}x_{2} & -x_{1} & u_{6}v_{1} & x_{0}+v_{2} \\
v_{5} & 0 & -x_{1} & -x_{2} \\
-x_{0} & v_{1} & 0 & 0 \end{array}\right), $$

$$ \nu' = \left(\begin{array}{c}
v_{1} \\
x_{0} + v_{2} \\
-x_{2} \\
x_{1} \end{array}\right)$$ 

Differentiating the above in terms of $u_{ij}$, we get the deformation corresponding to $\frac{\partial}{\partial u_{ij}}$. Note that $\lambda'\cdot\mu' \equiv \mu'\cdot\nu' \equiv 0$ mod $(v_{1}v_{2}, v_{1}v_{3}, v_{1}v_{4}, v_{1}v_{5})$ i.e. mod $(u_{1i}u_{2j}, \ldots, u_{1i}u_{5j})$ for $3\leq i,j\leq n$
and no higher order terms arise in these relations. Hence, the versal deformation space of $I$ is isomorphic to
$$\mbox{Spec} \big(k[u_{1i},\ldots, u_{5i}, u_{6}]/(u_{1i}u_{2j}, \ldots, u_{1i}u_{5j})\big).$$ 

To add the trivial deformations corresponding to $\frac{\partial}{\partial t_{ij}}$, we can take 
$$x_{0} = x_{0} + \sum_{i\geq 1}t_{0i}x_{i}, \ x_{1} = x_{1} + \sum_{i\geq 2}t_{1i}x_{i}, \ x_{2} = x_{2}+\sum_{i\geq 3}t_{2i}x_{i}. $$
Hence, the universal deformation of $I$ is given by
$$\mbox{Spec} \big(k[u_{1i},\ldots, u_{5i}, u_{6}, t_{ij} ]/(u_{1i}u_{2j}, \ldots, u_{1i}u_{5j})\big). $$ 
It is isomorphic to $\mathbb A^{4n-4}\cup \mathbb A^{7n-10}$, where $\mathbb A^{4n-4}$ has coordinates 
$u_{1i}, 3\leq i\leq n, u_{6}, t_{ij}, 0\leq i\leq 2, i <j\leq n$ and $\mathbb A^{7n-10}$ has coordinates 
$u_{ij}, 2\leq i\leq 5, 3\leq j\leq n, u_{6},  t_{ij}, 0\leq i\leq 2, i <j\leq n$. They intersect transversely along $\mathbb A^{3n-2}$, 
whose coordinates are given by $u_{6},  t_{ij}, 0\leq i\leq 2, i <j\leq n$. 
\end{proof}

Now we are ready to prove Theorem~\ref{1} (2). 

\begin{proof}[Proof of Theorem~\ref{1} (2)] 
By Proposition~\ref{deformation}, the universal deformation space of the ideal $I$ of $X$ has two components. The first $\mathbb A^{4n-4}$ corresponds to the deformation of $X$ along 
the $(4n-4)$-dimensional Hilbert component $H_n$. The second $\mathbb A^{7n-10}$ implies there exists another Hilbert component of dimension 
at most $7n-10$, which also contains the locus (IV). Below we will describe that component. 

Consider the incidence correspondence $\Sigma = \{(Q, \Lambda_{n-3}, \Lambda_{n-4})\}$, where $Q$ is a quadric $(n-2)$-fold in $\Pn$, $\Lambda^{k}$ denotes a 
$k$-dimensional linear subspace and $\Lambda_{n-4} = Q\cap \Lambda_{n-3}$. Using the projection $\Sigma \rightarrow \mathbb G(n-4, n)$, $\Sigma$ is irreducible and has dimension $7n-10$. For an element parameterized by $\Sigma$, let $X'$ be the subscheme $Q\cup \Lambda_{n-3}$ in $\Pn$. By the exact sequence 
$$0\rightarrow \mathcal O_{X'}\rightarrow \mathcal O_{Q}\oplus \mathcal O_{\Lambda_{n-3}}\rightarrow \mathcal O_{\Lambda_{n-4}}\rightarrow 0,$$
the Hilbert polynomial of $X'$ equals $P_n$. Hence, there is a component $H_n'$ of Hilb$^{P_n}(\Pn)$ that parameterizes $X'$. We will show that $X'$ can specialize to the subschemes of type (III). Without loss of generality, 
assume the hyperplane spanned by $Q$ is $x_{0} = 0$ and $\Lambda_{n-3}$ is given by $x_{1} = x_{2} = x_{3} = 0$. Then $\Lambda_{n-4}$ is specified 
by $x_{0} = x_{1} = x_{2} = x_{3} = 0$. Let $Q$ degenerate to a pair of codimension two linear subspaces in $x_{0} = 0$, say, 
it has ideal $(x_{0}, x_{1}x_{2})$. Then let $\Lambda_{n-3}$ approach the intersection $x_{0} = x_{1} = x_{2} = 0$ of the two subspaces by writing 
the ideal of $\Lambda_{3}$ as $(x_{1}, x_{2}, tx_{3}+(1-t)x_{0})$. For general $t$, the union of $Q$ and $\Lambda_{n-3}$ has ideal 
$(x_{0}, x_{1}x_{2})\cdot (x_{1}, x_{2}, tx_{3}+(1-t)x_{0})$. As $t\rightarrow 0$, we see the limit ideal must contain 
$(x_{0}^{2}, x_{0}x_{1}, x_{0}x_{2}, x_{1}x_{2}) = (x_{0}, x_{1}x_{2})\cap (x_{0}, x_{1}, x_{2})^{2}$, which defines a subscheme with Hilbert polynomial $P_n$ parameterized by the locus (III). Hence, the limit ideal equals $(x_{0}^{2}, x_{0}x_{1}, x_{0}x_{2}, x_{1}x_{2})$ and $H_n'$ contains the locus (III). Geometrically, the approaching direction of $\Lambda_{n-3}$ provides the embedded structure supported on the codimension three linear subspace $x_{0} = x_{1} = x_{2} = 0$. 
Since subschemes of type (III) can specialize to type (IV), we know $H_n'$ also contains the locus (IV). 

Since a subscheme of type (II) does not possess a $(n-3)$-dimensional component, $H_n'$ does not intersect the locus (II). Hence, $H_n\cap H_n'$ consists of 
(III) and (IV). Because the two components in the universal deformation space of a subscheme of type (IV) intersect transversely and (IV) is the specialization of (III), we know $H_n$ and $H_n'$ intersect transversely along (III) $\cup$ (IV).  
\end{proof}

\begin{remark}
\label{number}
The number of irreducible components of Hilb$^{P_n}(\Pn)$ may increase rapidly with $n$. 

For $n=3$, there are only two components $H_3$ and $H'_3$. 

For $n=4$, there is one more component $H''_4$, whose general points parameterize a quadric surface $Q$ and a line $L_{0}$ intersecting at two points along with an isolated point $q$. Since $H_4$ does not intersect $H''_4$, by Hartshorne's connectedness theorem \cite{Ha}, $H'_4$ necessarily intersects $H''_4$. We can see how they intersect as follows.
On the one hand, for a quadric surface $Q$ and a line $L$ intersecting at a point in $\mathbb P^{4}$ parameterized by $H'_4$, let $L$ degenerate to $L_{0}$, which intersects $Q$ at two points. Then an embedded point will arise at an intersection point $p$ to make the Hilbert polynomial correct. On the other hand, let the isolated point $q$ approach $p$, whose limit will also yield the embedded structure at $p$. This shows how $H'_4$ and $H''_4$ intersect. 

For $n=5$, in addition to $H_5$ and $H'_5$, there are at least four other components of the full Hilbert scheme. The general points of the first one parameterize a quadric three-fold and a plane in a four dimensional linear subspace $\Lambda_{4}\subset \mathbb P^{5}$, along with a line intersecting the quadric at a point. The second parameterizes a quadric three-fold and a plane in $\Lambda_{4}$, along with a line intersecting the plane at a point. The third parameterizes a quadric three-fold, a plane and a line in $\mathbb P^{4}$, along with an isolated point. The last one parameterizes a quadric three-fold, a plane and a line in $\mathbb P^{4}$ such that the line intersects the plane, along with two isolated points. 

For $n=6$, in the same way one can list more than twenty components.  
\end{remark}

Now we prove Theorem~\ref{1} (3). 

\begin{proof}[Proof of Theorem~\ref{1} (3)]
By Proposition~\ref{deformation}, the universal deformation space of a subscheme of type (IV) is isomorphic to
$\mathbb A^{4n-4}\cup \mathbb A^{7n-10}$, where the two components correspond to deformations along $H_n$ and $H'_n$, respectively. 
This shows the $(4n-4)$-dimensional component $H_n$ and the $(7n-10)$-dimensional component $H'_n$ are both smooth along
the locus (IV). Since the closures of (I), (II) or (III) all contain (IV), $H_n$ is smooth everywhere. 
By Zariski's Main Theorem, the bijective morphism $\delta$ in Proposition~\ref{bijective} is an isomorphism. 

For $n=3$, by \cite[1.b]{H} and \cite[3.5, 4.2]{L}, Hilb$^{2m+2}(\Pth)$ consists of two components $H_3$ and $H'_3$. The second component $H'_3$ parameterizes a conic union an isolated point. In \cite[Theorem 3.5.1]{L}, a similar bijective morphism as Proposition~\ref{bijective} was established, $\sigma$: Bl$_{\Sigma}(\Pth\times \mbox{Hilb}^{2m+1}(\Pth))\rightarrow H'_3$, where $\Sigma$ denotes the incidence correspondence $\{p\in C\}$
for a point $p$ and a conic $C$. In order to show $\sigma$ is an isomorphism, it suffices to prove the smoothness of 
$H'_3$. A subscheme $C$ parameterized by $H'_3$ can specialize to a planar double line with an embedded point. If the embedded point is spatial, this is of type (IV). We have seen in the previous paragraph that $H'_3$ is smooth along (IV).

Let $k[x,y,z,w]$ denote the coordinate ring of $\Pth$. If the embedded point is also in that plane, the ideal of $C$ is equivalent to $I = (z, xy^{2}, y^{3})$. Let $\mathcal N$ and $\mathcal N'$ denote the normal sheaves of $C$ in $\Pth$ and in the plane $z=0$, respectively.
Let $S' = k[x,y,w]$ and $I'= (xy^{2}, y^{3})$ be the ideal of $C$ in $S'$. We want to show $h^{0}(\mathcal N) = 11$. By the exact sequence 
$$ 0\rightarrow \mathcal N'\rightarrow \mathcal N\rightarrow \OO_{C}(1)\rightarrow 0 $$ 
and $h^{0}(\OO_{C}(1)) = 4$, it suffices to show $h^{0}(\mathcal N') = 7$. One checks that the condition of 
Theorem~\ref{comparison} holds for $C$ regarded as a subscheme of the plane $z=0$. Then we only need to verify 
dim Hom$_{S'}(I', S'/I')_{0} = 7$. 
An element $\phi$ in Hom$_{S'}(I', S'/I')_{0}$ satisfies $y\phi(xy^{2}) = x\phi(y^{3})$ modulo $I'$. 
One checks that 
$$\phi(xy^{2}) = a_{1}x^{2}y + a_{2}x^{3} + a_{3}w^{2}x + a_{4}wx^{2} + a_{5}wxy+ a_{6}wy^{2},$$ 
$$\phi(y^{3}) = a_{2}x^{2}y + a_{3}w^{2}y + a_{4}wxy + a_{7}wy^{2}$$ with 7 parameters $a_{1},\ldots, a_{7}$ 
generate Hom$_{S'}(I', S'/I')_{0}$. Hence, dim Hom$_{S'}(I', S'/I')_{0} = 7$ and $h^{0}(\mathcal N) = 11$, which implies that the tangent space of $H'_3$ at $[C]$ has dimension 11. Combining with the previous paragraph, we know that $H'_3$ is smooth everywhere, so it is isomorphic to Bl$_{\Sigma}(\Pth\times \mbox{Hilb}^{2m+1}(\Pth))$. 
\end{proof}

Our complete analysis of Hilb$^{2m+2}(\Pth)$ extends Hilb$^{2m+2}(\Pn)$ with $n\geq 4$. 

\begin{corollary}
\label{n>3}
For $n\geq 4$, Hilb$^{2m+2}(\Pn)$ consists of two components $W_{n}$ and $W_{n}'$. The component $W_{n}$ has dimension $4n - 4$ and its general point parameterizes a pair of skew lines. The component $W_{n}'$ has dimension $4n-1$ and its general point parameterizes a conic union an isolated point. Both $W_{n}$ and $W_{n}'$ are smooth. They intersect transversely along a $(4n-5)$-dimensional locus $E_n$ whose general points parameterize a pair of coplanar lines with a spatial embedded point at their intersection. In particular, $W_{n}$ is an $H_3$ bundle over $\mathbb G(3,n)$.  
\end{corollary}

\begin{proof}
Using the arguments in \cite[Lemma 1]{PS} and \cite[Lemma 3.5.3]{L}, a subscheme $C$ in $\Pn$ with Hilbert polynomial $2m+2$ is contained in a linear subspace $\Pth\subset \Pn$. Hence, Hilb$^{2m+2}(\Pn)$ has two components $W_{n}$ and $W_{n}'$, whose general points parameterize 
a pair of skew lines and a conic union an isolated point, respectively. We have dim $W_{n}$ = dim $\mathbb G(3,n)$ + dim $H_3$ = $4n-4$ and 
dim $W_{n}'$ = dim $\mathbb G(3,n)$ + dim $H_3'$ = $4n-1$. Moreover, $W_{n}$ and $W_{n}'$ intersect along the $(4n-5)$-dimensional locus $E_n$ whose general point parameterizes two incident lines with a spatial embedded point at their intersection. 

For $C\subset\Pth\subset \Pn$, let $\mathcal N$ and $\mathcal N'$ denote the normal sheaves of $C$ in $\Pn$ and $\Pth$, respectively.
By the exact sequence 
$$0\rightarrow \mathcal N'\rightarrow \mathcal N\rightarrow \mathcal N_{\Pth/\Pn}|_{C}\rightarrow 0,$$ 
we get $h^{0}(\mathcal N) = h^{0}(\mathcal N') + (n-3)h^{0}(\OO_{C}(1)) = h^{0}(\mathcal N') + 4(n-3)$. 
Since $h^{0}(\mathcal N')$ equals the dimension of the tangent space of Hilb$^{2m+2}(\Pth)$ at $[C]$, we know that 
$h^{0}(\mathcal N)$ equals the dimension of $W_{n}$ or $W_{n}'$ for $[C]\in W_{n}\backslash E_{n}$ or $W_{n}'\backslash E_{n}$, respectively. 
Hence, $W_{n}\backslash E_{n}$ and $W_{n}'\backslash E_{n}$ are smooth. 

For $[C]\in E_{n}$, note that $C$ spans $\Pth$. The deformation of $C\subset \Pth$ in Proposition~\ref{deformation}
along with the deformation corresponding to perturbing $\Pth$ in $\Pn$ provide a $4n$ dimensional universal deformation space 
for $C\subset \Pn$. This space is isomorphic to $\mathbb A^{4n-4}\cup \mathbb A^{4n-1}$, where $\mathbb A^{4n-4}\cap \mathbb A^{4n-1}=\mathbb A^{4n-5}$.  
This shows that $W_{n}$ and $W_{n}'$ are smooth along $E_n$ and they intersect transversely.  

Finally, a subscheme $C$ parameterized by $W_{n}$ uniquely determines a $\Pth$ spanned by $C$. So $W_{n}$ admits a fibration over $\mathbb G(3,n)$ with fiber isomorphic to $H_3$. In contrast, $W_{n}'$ does not admit a natural fibration over $\mathbb G(3,n)$, since a plane conic with a point on that plane only span $\Pt$ rather than $\Pth$. 
\end{proof}

\section{Mori theory of $H_n$}
In this section, we will prove Theorems~\ref{2} and \ref{3}. To study the geometry of the divisors defined in Definition \ref{divisors}, 
we calculate their intersection numbers with the following test curves. 

\begin{definition}
\label{curves}
We introduce effective curves in $H_n$ as follows. 

Let $B_{1}$ denote a pencil of codimension two linear subspaces contained in a hyperplane union a fixed general codimension two linear subspace in $\Pn$. 

Let $B_{2}$ denote a pencil of codimension two linear subspaces contained in a hyperplane union a fixed codimension two subspace in this pencil. Put an embedded structure at the base $\mathbb P^{n-3}$ of the pencil given by the square of its ideal. 

Take a pencil of lines from a ruling class of a quadric surface in $\Pth$. Each line along with a fixed line in that ruling class and
a fixed codimension four subspace can span a pair of codimension two linear subspaces in $\Pn$. Let $B_{3}$ denote this family in $H_n$. 

Let $B_{4}$ denote a pencil of subschemes defined by the ideal $(x_{0}^{2}, x_{0}x_{1}, x_{1}^{2}, tx_{0}x_{3}-sx_{1}x_{2})$, where $[s,t]$ denote the coordinates of $\mathbb P^{1}$. 
\end{definition}

\begin{lemma}
\label{intersection}
We have the following intersection numbers: 
$$B_{1}\ldotp M = 1, \ B_{1}\ldotp N = 0, \ B_{1}\ldotp F = 1, \ B_{1}\ldotp E = 1, $$
$$B_{2}\ldotp M = 1, \ B_{2}\ldotp N = 2, \ B_{2}\ldotp F = 0, $$
$$B_{3}\ldotp M = 2, \ B_{3}\ldotp N = 2, \ B_{3}\ldotp E = 0, $$
$$B_{4}\ldotp M = 0, \ B_{4}\ldotp F = 1, \ B_{4}\ldotp E = 2. $$
\end{lemma}

\begin{proof}
Let us verify the intersection numbers involving $B_{1}$. The others can be checked similarly. 

Suppose $B_{1}$ is given by a pencil of codimension two linear subspaces $(x_0, sx_1 + tx_2)$ union a fixed general 
codimension two linear subspace $\Lambda_{n-2}$ defined by $(x_0-x_3, x_1+x_3)$. The pencil has a base codimension three linear subspace 
$\Lambda_{n-3}$: $x_0 = x_1 = x_2 = 0$. 

Take a line $L$ that defines $M$ whose ideal is $(x_2, \ldots, x_n)$. There is a unique subscheme with $[s,t] = [0,1]$ in $B_{1}$ intersecting with $L$. 
To check that $B_1$ intersects $M$ transversely, around $(x_0, x_2)\cap(x_0-x_3, x_1+x_3)$, subschemes in $H_n$ have ideal 
$(x_0+a_1x_1+\sum_{i=3}^n a_ix_i, x_2+b_1x_1+\sum_{i=3}^n b_ix_i)\cap (x_0-x_3 + \sum_{i=2}^n c_ix_i, x_1+x_3+\sum_{i=2}^n d_ix_i)$, where 
$a_i, b_i, c_i, d_i$ yield a local chart for $H_n$. The divisor $M$ corresponds to the locus $b_1=0$. The pencil $B_1$ corresponds to the locus 
where all $a_i, b_i, c_i, d_i$ are zero except $b_1$. Hence, $B_1$ intersects $M$ transversely at their unique meeting point, so $B_{1}\ldotp M = 1$. 

Since $\Lambda_{n-2}$ is not in the pencil, there is no generically non-reduced subscheme parameterized by $B_1$. So $B_{1}$ does not intersect $N$. 

For $E$, take its defining $\Lambda_{3}$ with ideal $(x_4, \ldots, x_n)$. There is a unique subscheme in $B_{1}$ with $[s,t] = [1,0]$ intersecting $\Lambda_3$.
To check that $B_1$ intersects $E$ transversely, around $(x_0, x_1)\cap(x_0-x_3, x_1+x_3)$, subschemes in $H_n$ have ideal 
$(x_0+\sum_{i=2}^n a_ix_i, x_1+\sum_{i=2}^n b_ix_i)\cap (x_0-x_3 + \sum_{i=2}^n c_ix_i, x_1+x_3+\sum_{i=2}^n d_ix_i)$, where 
$a_i, b_i, c_i, d_i$ yield a local chart for $H_n$. The pencil $B_1$ corresponds to the locus where all $a_i, b_i, c_i, d_i$ are zero except $b_2$.
The divisor $E$ corresponds to the locus $(a_3+1)(b_2 - d_2) = (a_2-c_2)(b_3-1)$. Hence, $B_1$ and $E$ intersect transversely, so $B_1\ldotp E = 1$.  

Take a general point-plane flag $(q \in \Lambda_{2})$ that defines $F$. Suppose $\Lambda_{2}$ intersects $\Lambda_{n-2}$ at a point $r$. The line 
$\overline{qr}$ intersects a unique codimension two subspace in the pencil. As above, one checks that $B_1$ and $F$ intersect transversely, so $B_{1}\ldotp F = 1$. 
\end{proof}

\begin{proof}[Proof of Theorem~\ref{2} (1)]
Since $H_n$ is a smooth rationally connected variety, the fundamental group $\pi_1 (H_n)$ is trivial, hence $H_1 (H_n, \mathbb Z)$ is also trivial, cf. e.g. \cite{KMM}. Using the universal coefficient theorem for cohomology involving the Ext functor, $H^2 (H_n, \mathbb Z)$ is torsion free. Since $H^1 (\OO_{H_n})$ is trivial, Pic($H_n$) embeds into $H^2 (H_n, \mathbb Z)$ as a subgroup. So Pic($H_n$) is torsion free. 

Recall that $N$ is the exceptional divisor of the blow-up of Sym$^2 \Gn$ along the diagonal. Then $H_n\backslash N$ is isomorphic to Sym$^2 \Gn \backslash \Delta$, whose divisor class group has rank one. Hence, Pic($H_n$) is a rank two free $\mathbb Z$ module. By the intersection numbers $B_2\ldotp M = 1, B_2\ldotp F = 0$ and $B_4\ldotp M = 0, B_4\ldotp F = 1$, cf. Lemma~\ref{intersection}, the divisor classes $M$ and $F$ generate Pic($H_n$). 

Let $Null\subset$ Pic($H_n)$ be the null space with respect to the intersection pairing. Since Pic($H_n$) is generated by $M$ and $F$, 
a divisor class $D\in Null$ is linearly equivalent to $aM+bF$ for $a, b\in \mathbb Z$. Using the test curves $B_{2}$ and $B_{4}$, we 
get $a = b = 0$. Hence, Num($H_n$) = Pic($H_n)/Null$ is isomorphic to Pic($H_n$). 

Using the test curves $B_{1}$ and $B_{2}$, we get $N= 2M-2F$. Using $B_{1}$ and $B_{3}$, we get $E=M-N= 2F-M$. 
\end{proof}

These divisors decompose the effective cone of $H_n$ as follows: 

\begin{figure}[H]
    \centering
    \psfrag{M}{$M$}
    \psfrag{N}{$N$}
    \psfrag{E}{$E$}
    \psfrag{F}{$F$}
    \includegraphics[scale=0.6]{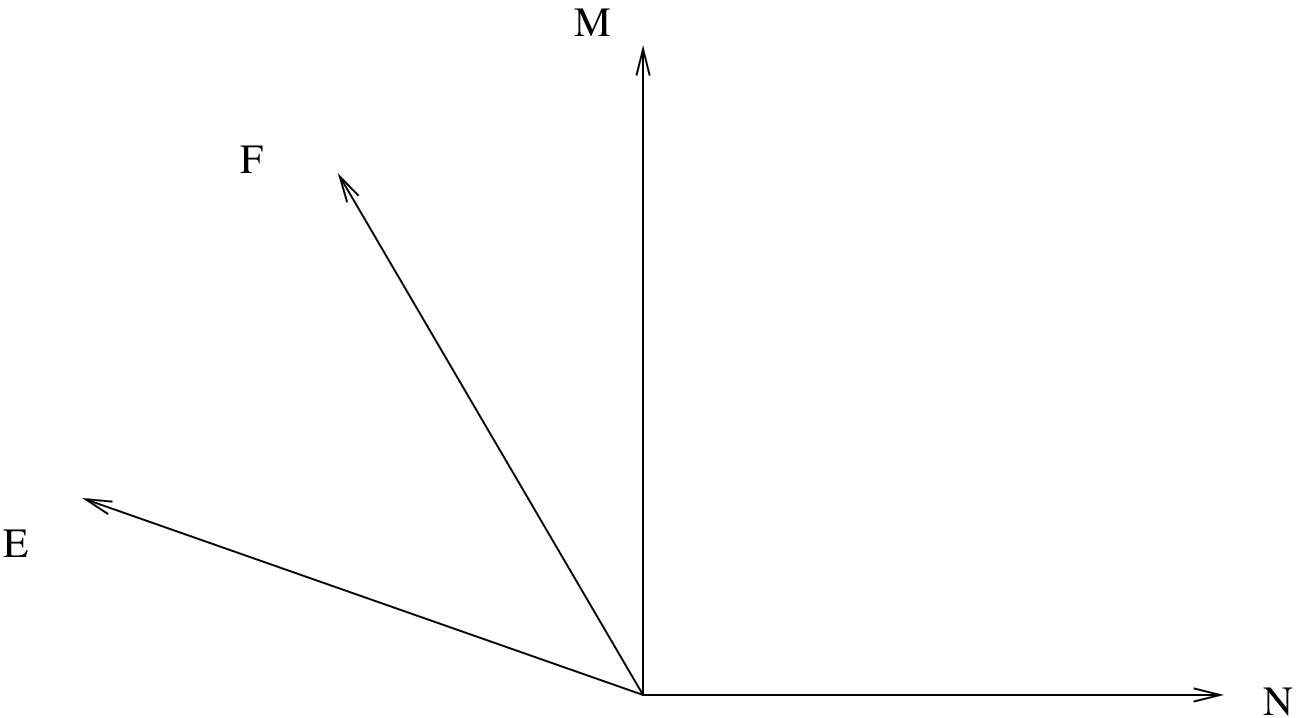}
\end{figure}

If a divisor $D$ has negative intersection with an irreducible curve, then its stable base locus $\BD$ necessarily contains the locus swept out by the deformations of this curve. 

\begin{proof}[Proof of Theorem~\ref{2} (2)]
By the definition, $M$ and $F$ are base-point-free, since we can perturb their defining lines and flags to avoid any point in $H_n$. So $(F, M)$ is the ample cone. 
 
The divisor $N$ parameterizes the loci (II) $\cup$ (IV), which is the exceptional divisor of the blow-up of Sym$^{2}\Gn$ along the diagonal. 
Hence, $N$ spans an extremal ray of the effective cone. Note that $B_{3}\ldotp E = 0$ and $B_{3}$ is a moving curve in $H_n$. Since Pic($H_n$) is of rank two, $E$ spans another extremal ray of the effective cone. 

Note that $B_{4}\ldotp M = 0$ and $B_{4}\ldotp N = B_{4}\ldotp (2M-2F) = -2$, which implies that $B_{4}$ is contained in $\BD$ for a divisor $D$ in the chamber $(M, N]$. Since $B_{4}$ is a moving curve in $N$ and $M$ is base-point-free, we know $\BD$ consists of $N$. 

For a divisor class $D$ in the chamber $[E, F)$, since $B_{2}\ldotp F = 0$ and $B_{2}\ldotp E = B_{2}\ldotp (M-N) = -1$, 
we have $B_{2}\ldotp D < 0$. Since $B_{2}$ sweeps out the loci (III) $\cup$ (IV), $\BD$ contains (III) $\cup$ (IV). 
Note that $F$ is base-point-free. For a subscheme $X$ not parameterized in (II) $\cup$ (III) $\cup$ (IV), we can choose the defining $\Pth$ of $E$ away from the intersection of the two subspaces of $X$. Since (II) is a divisorial locus parameterized by $N$ and 
any two subschemes in (II) are projectively equivalent, it implies the base locus of $E$ does not intersect (II). Hence, $\BD$ consists of (III) $\cup$ (IV) for 
$D$ in the chamber $[E, F)$. 
\end{proof}

Let us calculate the canonical class of $H_n$. 

\begin{proposition}
\label{KH}
Let $n\geq 3$ be an integer. The canonical divisor $K_{H_n}$ has class $ -(n+1)M + (n-2)N$. 
In particular, $H_n$ is Fano iff $n = 3$ or $4$. 
\end{proposition} 

\begin{proof}
Let $Y = \Gn \times \Gn$ and $Y'$ be the blow-up of $Y$ along its diagonal. Let $E_0$ denote the exceptional divisor 
of the blow-up. Since $H_n\cong \mbox{Bl}_\Delta\mbox{Sym}^2\Gn$, we calculate the canonical class $K_{H_n}$ via the following commutative diagram: 
$$\begin{CD}
Y' @>g>>   \mbox{Bl}_\Delta\mbox{Sym}^2\Gn \\
@V{\operatorname{\phi}}VV      @VV{\operatorname{\rho}}V \\
Y @>f>> \mbox{Sym}^2 \Gn  
\end{CD}$$

Note that $g$ is a double cover branched along $N$, and $E_0$ is the ramification divisor. By the Riemann-Hurwitz formula, $K_{Y'} = g^{*}K_{H_n} + E_0$. 
By the blow-up formula, $K_{Y'} = \phi^{*}K_{Y} + (2n-3)E_0$. Hence, $g^{*}K_{H_n} = \phi^{*}K_{Y} + (2n-4)E_0.$ 
The canonical class $K_Y$ is equivalent to $\OO_Y(-n-1, -n-1)$. Moreover, $\phi^{*}\OO_Y(1,1) = g^{*}M$ 
and $g^{*}N = 2E_0$. Suppose $K_{H_n} = aM + bN$. Then $ag^{*}M + 2b E_0 = -(n+1)g^{*}M + (2n-4)E_0$, 
which implies $a = -(n+1)$ and $b = n-2$. Therefore, $-K_{H_n} =  (n+1)M - (n-2)N = 2(n-2)F + (5-n)M$, which lies in the ample cone of $H_n$ 
iff $2<n<5$, cf. Theorem~\ref{2} (2). 
\end{proof}

Now we consider the Proj model $P(D)$ induced by a divisor $D$. The following result will be used frequently, cf. e.g., \cite[2.1.B]{La}. 

\begin{lemma}
\label{ample}
Let $f$: $X\rightarrow Y$ be a birational morphism between two normal varieties. Let $D$ be an ample divisor on $Y$. 
Then $f^{*}D$ is semi-ample on $X$ and the Proj model $P(f^{*}D)$ is isomorphic to $Y$. 
\end{lemma}

As mentioned in Theorem~\ref{3}, the case $n=3$ is slightly different from $n\geq 4$, since 
the locus (III) is divisorial iff $n=3$. Therefore, we first study $H_3$. This may also help the reader get a feel for the models. 

\begin{proof}[Proof of Theorem~\ref{3} for $n=3$] 
Let $\mbox{Sym}^{2}\mathbb G(1,3)$ be the Chow variety parameterizing cycles $[L_{1}+L_{2}]$, where 
$L_{1}$ and $L_{2}$ are two lines in $\Pth$. Let $M_0$ be the divisor class parameterizing cycles in the Chow variety
that intersect a fixed line. Then $M_0$ yields the defining ample line bundle for the Chow variety, cf. \cite[1.a]{H}. 
The Hilbert-Chow morphism $H_{3}\rightarrow \mbox{Sym}^{2}\mathbb G(1,3)$ pulls $M_0$ back to $M$. Moreover, $\mbox{Sym}^{2}\mathbb G(1,3)$ has finite quotient singularities, hence is normal. By Lemma~\ref{ample}, the model $P(M)$ is isomorphic to $\mbox{Sym}^{2}\mathbb G(1,3)$. The locus of double lines supported on a reduced line $L$ gets contracted 
to a point parameterizing the cycle $[2L]$ in Sym$^{2}\mathbb G(1,3)$. 

Consider the Pl\"{u}cker embedding of $\mathbb G(1,3)$. The image is a smooth quadric 4-fold $Q$ in $\mathbb P^5$. 
Recall in Definition~\ref{spaces} that $\Phi_3\cong\mathbb G(3,5)$ parameterizes codimension two linear sections of $Q$. 
Define a morphism $f$: $H_3\rightarrow \mathbb G(3,5)$ by sending a subscheme $X$ to the locus of lines whose intersections with $X$ have length $\geq 2$.  
Let us check that this locus is a codimension two linear section of $Q$. 

If $X$ is a pair of skew lines $L_{1}\cup L_{2}$, the space of lines in $\Pth$ that intersect both $L_{1}$ and $L_{2}$ is a smooth quadric surface contained in $Q$, 
which is cut out by a general $3$-dimensional linear subspace of $\mathbb P^{5}$. 
If $X$ is a double line without embedded point, its Zariski tangent space $T_{q}X$ at a point $q$ is 2-dimensional. Consider a line passing through $q$ whose schematic intersection with $X$ has length at least two. This line has to be contained in $T_{q}X$. The space of such lines forms a 2-dimensional quadric cone in $X$. The cone point parameterizes the line $X_{red}$. A ruling through the cone point parameterizes lines passing through $q$ and contained in $T_{q}X$. 
If $X$ consists of two lines $L_{1}$ and $L_{2}$ contained in a plane $\Lambda$ with an embedded point at their intersection $p$, the space of lines in $\Pth$ that intersect $L_{1}$ and $L_{2}$ is a union of two planes contained in $Q$. One plane is the Schubert variety $\Sigma_{1,1}$ parameterizing lines contained in $\Lambda$ and the other is the Schubert variety $\Sigma_2$ parameterizing lines passing through $p$. The two planes intersect along a 1-dimensional linear subspace parameterizing lines contained in $\Lambda$ and passing through $p$. Note that these two planes are determined by the flag $(p\in \Lambda)$ and independent of the two lines $L_{1}$, $L_{2}$. Such two planes or quadric cones in $Q$ are cut out by special $3$-dimensional linear subspaces of $\mathbb P^{5}$. Therefore, $f$ is well-defined and is a surjective morphism. 

The family $B_{2}$ in Lemma~\ref{intersection} is a moving curve in $E$ and has zero intersection with $F$. 
For a subscheme $X$ parameterized by $B_{2}$, we have seen that the space of lines in $\Pth$ whose intersections with $X$ have length at least two does not depend on $X$. It is only determined by the embedded point and the plane containing $X_{red}$. Hence, the morphism $f$ contracts $B_2$ to the point in $\mathbb G(3,5)$ corresponding to the linear section $\Sigma_{1,1}\cup\Sigma_2$. Let $\sigma_1$ be the hyperplane class of $\mathbb G(1,3)$. Using the test curve $B_{1}$ in Lemma~\ref{intersection}, we have $B_{1}\ldotp f^{*}\sigma_1 = f_{*}B_{1}\ldotp \sigma_1 = 1$. Since $F\ldotp B_1 = 1, F\ldotp B_2 = 0$ and the Picard number of $H_3$ is two, we get $F = f^{*}\sigma_1$, which implies that the model $P(F)$ is isomorphic to $\mathbb G(3,5)$.  
\end{proof}

Recall in Corollary~\ref{n>3}, there is a smooth Hilbert component $W_n$ whose general point parameterizes a pair of skew lines in $\Pn$. The geometry of $H_3 \cong W_3$ serves as a prototype 
for that of $W_n$. We can similarly define effective divisors on $W_{n}$ as follows. 
\begin{definition}
Let $n\geq 4$ be an integer. 

Let $M'$ denote the divisor class parameterizing the locus of subschemes whose supports intersect a fixed codimension two linear subspace. 

Let $N'$ denote the divisor class parameterizing the locus of double lines. 

Let $E'$ denote the divisor class parameterizing the locus of two coplanar lines with a spatial embedded point at their intersection. 

Let $R'$ denote the divisor class parameterizing the locus of subschemes such that the 3-dimensional linear subspaces they span intersect a fixed codimension four linear subspace.  

Fix a flag $\Lambda_{n-3}\subset\Lambda_{n-1} \subset \Pn$. For a pair of general lines, let $p, q$ denote their intersection points with $\Lambda_{n-1}$. Consider the locus of two lines such that $p, q$ and $\Lambda_{n-3}$ only span a codimension two linear subspace. Denote by $F'$ the divisor class parameterizing the closure of this locus. 
\end{definition}

\begin{remark}
There is a rational map $W_n\dashrightarrow H_3$ by projecting a subscheme from a codimension four linear subspace to a linear subspace $\Pth\subset \Pn$. Then $M', N', E', F'$ on $W_n$ are equivalent to the pull-backs of $M, N, E, F$ from $H_3$, respectively.   
\end{remark}

Since $H_3$ naturally embeds into $W_n$ via the inclusion $\Pth\subset \Pn$, we can adapt the test curves $B_{1}, \ldots, B_{4}$ in Definition~\ref{curves} and their intersection numbers in Lemma~\ref{intersection} to $W_n$ for $n\geq 4$. Two more test curves are needed as follows. 

\begin{definition}
\label{Wn-curves}
Let $n\geq 4$ be an integer. 

Take a line $L$ away from a fixed plane in $\Pn$. Let $B_{5}$ denote a pencil of lines on that plane union $L$ as a one parameter family in $W_n$. 

Fix two lines $L_1$ and $L_2$ meeting at a point $p$ on a plane $\Lambda$. Take a line $L$ away from $\Lambda$. For each point $q\in L$, there is a unique spatial embedded structure at $p$ such that the 3-dimensional linear subspace it spans with $\Lambda$ contains $q$. Varying $q$, we get a one parameter family $B_6$ in $W_n$. 
\end{definition}

Now we study the stable base locus decomposition of the effective cone of $W_n$. 

\begin{proposition}
\label{Wn-chambers}
Let $n\geq 4$ be an integer. On $W_n$, we have $E' = 2F' - M' - R'$ and $N' = 2M' - 2F'$. The effective cone of $W_{n}$ is generated by $R', E', N'$ and the semi-ample cone of 
$W_{n}$ is generated by $R', F', M'$. For a divisor $D$ in the chambers $\langle E', F', R'\rangle \cup \langle E', F', M'\rangle$, $\BD$ consists of $E'$. 
For $D$ in the chamber $\langle R', M', N'\rangle$, $\BD$ consists of $N'$. For $D$ in the chamber $\langle E', M', N' \rangle$, $\BD$ consists of 
$E'$ and $N'$. 
\end{proposition}

\begin{proof}
Since $W_n$ admits an $H_3$ fibration over $\mathbb G(3,n)$, its Picard number equals three. 
Using $B_{1}, B_{2}$ and $B_{3}$, we get $E' = 2F'-M' + aR'$ and $N' = 2M'-2F'+bR'$, since $B_{i}\ldotp R' = 0$ for $1\leq i\leq 3$. Using the curve $B_{5}$, we have $B_{5}\ldotp M' = B_{5}\ldotp F' = B_{5}\ldotp R' = 1$ and $B_{5}\ldotp N' = B_{5}\ldotp E' = 0.$
Therefore, we get $E' =  2F' - M' - R'$ and $N' = 2M' - 2F'$. 

$M', F'$ and $R'$ are base-point-free by their definitions, since we can vary their defining linear subspaces or flags to make their loci avoid any point in $W_{n}$. So they span the semi-ample cone. 

Since $B_{2}\ldotp E' = -1$, $B_{2}\ldotp F' = B_{2}\ldotp R' = 0$ and $B_{2}$ is a moving curve 
in $E'$, $E'$ spans an extremal ray of the effective cone and it is the stable base locus of a divisor in the chamber $\langle E', F', R'\rangle$.
We have $B_{6}\ldotp M' = B_{6}\ldotp F' = 0, \ B_{6}\ldotp R' = 1$, so $B_{6}\ldotp E' = -1$. This implies $E'$ is the stable base locus for a divisor in the chamber $\langle E', F', M'\rangle$. 

Since $B_{4}\ldotp N' = -2, \ B_{4}\ldotp R' = B_{4}\ldotp M' = 0$ and $B_{4}$ is a moving curve in $N'$, $N'$ spans 
an extremal ray of the effective cone and it is the stable base locus of a divisor in the chamber $\langle M', N', R'\rangle$.

Finally, $B_{6}\ldotp E' = -1$ and $B_{6}\ldotp M' = B_{6}\ldotp N' = 0$ imply that a divisor $D$ in the chamber $\langle E', N', M'\rangle$ contains 
$E'$ in its stable base locus. After removing $E'$, since $B_{4}\ldotp M' = 0$ and $B_{4}\ldotp N' = -2$,  
$N'$ is also contained in the base locus of $D$. Since $M'$ is base-point-free, it implies that the stable base locus of $D$ consists of $E'\cup N'$. 
\end{proof}

The picture below describes this stable base locus decomposition: 
 \begin{figure}[H]
    \centering
    \psfrag{R}{$R'$}
    \psfrag{E}{$E'$}
    \psfrag{F}{$F'$}
    \psfrag{M}{$M'$}
    \psfrag{N}{$N'$}
     \includegraphics[scale=0.5]{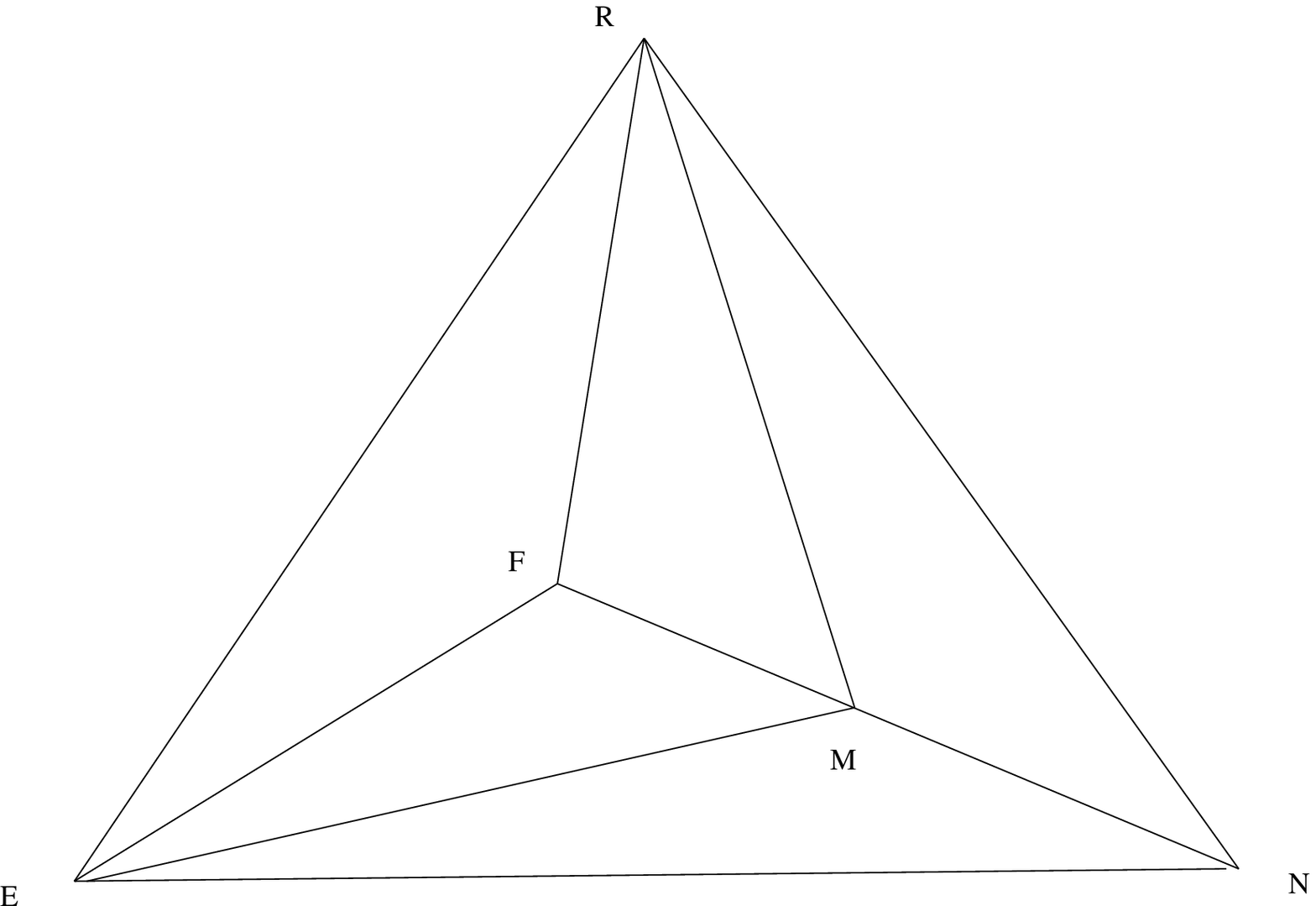}
 \end{figure}

Now we study the morphism $\psi_D$: $W_n\rightarrow P(D)$ induced by a divisor $D$ in the semi-ample cone $\langle R', F', M'\rangle$. 
Recall the spaces $\Psi_n$ and $\Theta_n$ in Definition~\ref{spaces}. 

\begin{proposition} 
\label{Wn-models}
Let $a, b$ denote two positive integers and $n\geq 4$ denote an integer. 

(1) For $D_{1} = aF' + bM'$, the morphism $\psi_{D_{1}}$ contracts $E'$ and $P(D_{1})$ is isomorphic to Bl$_{\Delta}$Sym$^{2}\Go$. 

(2) For $D_{2} = aF' + bR'$, the morphism $\psi_{D_{2}}$ contracts $E'$ and $P(D_{2})$ is isomorphic to $\Psi_n$.  

(3) For $D_{3} = aM' + bR'$, the morphism $\psi_{D_{3}}$ contracts $N'$ and $P(D_{3})$ admits a Sym$^{2}\mathbb G(1,3)$ fibration over 
$\mathbb G(3,n)$, which is the relative Chow variety parameterizing two lines in a 3-dimensional linear subspace.

(4) The morphism $\psi_{F'}$ contracts $E'$ and $P(F')$ is isomorphic to $\Theta_n$. Moreover, $P(D_{1})\dashrightarrow P(D_{2})$ is a flip over $P(F')$.

(5) The morphism $\psi_{M'}$ contracts $E'$ and $N'$. The model $P(M')$ is isomorphic to Sym$^{2}\Go$. 

(6) The morphism $\psi_{R'}$ is the $H_{3}$ fibration $W_n\rightarrow \mathbb G(3,n)$ and the model $P(R')$ is isomorphic to $\mathbb G(3,n)$. 
\end{proposition}

\begin{proof}
(1) On the one hand, by $B_{6}\ldotp F' = B_{6}\ldotp M' = 0$, the curve class $B_{6}$ is contracted by $\psi_{D_{1}}$. Since $B_{6}$ sweeps out $E'$, $\psi_{D_1}$ contracts $E'$. 

On the other hand, suppose an effective curve $C$ in $W_{n}$ does not intersect $M'$. Then $C\ldotp M' = 0$ implies the subschemes parameterized by $C$ have the same support. Otherwise we can choose a defining codimension two linear subspace for $M'$ to intersect with finitely many points of $C$ and $C\ldotp M'$ would be non-zero, a contradiction. Now, suppose $C$ does not intersect $F'$. If a general subscheme parameterized by $C$ consists of two skew lines or a non-planar double line, we can always choose a defining flag $\Lambda_{n-3}\subset \Lambda_{n-1}$ of $F'$ such that $F'$ intersects finitely many members of $C$. If a general subscheme parameterized by $C$ consists of two incident lines but the plane they span varies within the family, one can also choose a defining flag of $F'$ such that $F'$ only intersects finitely many members of $C$. Moreover, if the support of the spatial embedded point moves within $C$, one can still choose a defining flag of $F'$ such that $F'$ intersects finitely many members of $C$. In all these cases, $C\ldotp F'$ would be non-zero, a contradiction. So $C\ldotp F' = 0$ implies $C$ is a family of two incident lines contained in a common plane and both passing through a common point where the embedded points arise. 
Hence, if $C\ldotp M' = C\ldotp F' = 0$, the subschemes parameterized by $C$ have the same planar support and their embedded points also have the same support at a point $p$. Only the spatial embedded structure pointing outward the plane varies at $p$ in order to get the family $C$. 

Note that Bl$_{\Delta}$Sym$^{2}\Go$ parameterizes a pair of skew lines in $\Pn$, double lines of arithmetic genus $-1$ and two incident lines without specifying the embedded structure at their intersection point. By forgetting the embedded point of a subscheme parameterized by $E'$, the map $W_{n}\rightarrow$ Bl$_{\Delta}$Sym$^{2}\Go$ contracts the locus $E'$, and $B_{6}$ spans the contracted curve class. Pulling back an ample divisor from Bl$_{\Delta}$Sym$^{2}\Go$, we get a semi-ample divisor on $W_{n}$ whose intersection with $B_{6}$ is zero. So this divisor is of type $aF' + bM'$. Since Bl$_{\Delta}$Sym$^{2}\Go$ is smooth, by Lemma~\ref{ample}, it is isomorphic to the model $P(D_{1})$. 

(2) By $B_{2}\ldotp F' = B_{2}\ldotp R' = 0$, the curve class $B_{2}$ is contracted by $\psi_{D_{2}}$. 
Since $B_{2}$ sweeps out $E'$, $\psi_{D_2}$ contracts $E'$. 

Now suppose an effective curve $C$ in $W_{n}$ does not intersect $R'$. Those subschemes parameterized by $C$ must span the same $\Pth$. 
Otherwise we can choose a defining codimension four linear subspace of $R'$ such that $R'$ intersects finitely many points of $C$. Then $C\ldotp R'$ would be non-zero, a contradiction. If $C\ldotp F' = 0$, in (1), we analyzed that the 1-dimensional parts of subschemes parameterized by $C$ span the same plane, and 
the embedded points of those subschemes have the same support. Hence, if $C\ldotp F' = C\ldotp R' = 0$, then $C$ parameterizes a family of two incident lines 
in a common plane with the same spatial embedded structure supported on a common point. Only the 1-dimensional part of the two lines varies in that plane to get the family $C$. 

For a subscheme $X$ parameterized by $W_{n}$, it spans a unique $\Pth$. Associate to $X$ the closure of locus in $\mathbb G(1,3)$ of lines in that $\Pth$ whose intersection with $X$ is a length-2 zero dimensional subscheme. By Theorem~\ref{3} (4), we get a morphism from 
$W_{n}$ to $\Psi_n$, which is a $\mathbb G(3,5)$ bundle over $\mathbb G(3,n)$. This morphism restricted to each fiber $H_{3}$ only contracts the curve class $B_{2}$. Pulling back an ample divisor from the target, we get a semi-ample divisor on $W_{n}$ whose intersection with $B_{2}$ is zero. So this divisor is of type $aF'+bR'$ and its Proj model is isomorphic to $\Psi_n$. 

(3) By $B_{4}\ldotp M' = B_{4}\ldotp R' = 0$, the curve class $B_{4}$ is contracted by $\psi_{D_{3}}$. Since $B_{4}$ sweeps out $N'$, 
$\psi_{D_3}$ contracts $N'$. 

For an effective curve $C$ in $W_{n}$, by the analysis in (1) and (2), $C\ldotp M' = C\ldotp R' = 0$ implies that the subschemes parameterized by $C$ 
have the same support and span the same $\Pth$. Hence, they are double lines supported on a common line with the double structure varying to get the family $C$. 

For a subscheme $X$ parameterized by $W_{n}$, it spans a unique $\Pth$. Associate to $X$ its support as a cycle in that $\Pth$. 
By Theorem~\ref{3} (2) for $n=3$, we get a morphism from $W_{n}$ to a Sym$^{2}\mathbb G(1,3)$ bundle over $\mathbb G(3,n)$, which parameterizes a pair of linear cycles in a 3-dimensional linear subspace. This morphism restricted to each fiber $H_{3}$ of $W_{n}$ only contracts the curve class $B_{4}$. Pulling back an ample divisor from the target, we get a 
semi-ample divisor on $W_{n}$ whose intersection with $B_{4}$ is zero. So this divisor is of type $aM' + bR'$ and its Proj model is the relative Chow variety parameterizing two lines 
in a 3-dimensional linear subspaces. 

(4) Since $B_{2}\ldotp F' = B_{6}\ldotp F' = 0$ and the two curve classes sweep out $E'$, the map $\psi_{F'}$ contracts $E'$. In (1), we have seen that $C\ldotp F' = 0$ for an effective curve $C$ implies it is a family of two incident lines contained in a common plane $\Lambda$ and both passing through a common point $p$ where the embedded points arise. Namely, $\psi_{F'}$ forgets the lines and embedded structures within the family but only remembers the common flag $(p\in \Lambda)$. 

Define a morphism $f$: $W_n\rightarrow \Theta_n$ by sending a subscheme $X$ to the locus in $\Gn$ parameterizing codimension two linear subspaces whose intersections
with $X$ have length $\geq 2$. Let us check that this locus corresponds to a codimension two linear section of the Pl\"{u}cker embedding of $\Gn$ of type $\Sigma_1\cap\Sigma'_1$ or its degenerations. 

If $X$ consists of two skew lines $L$ and $L'$, this locus is the subvariety $\Sigma_1\cap \Sigma'_1$ in $\Gn$, where $\Sigma_1$ and $\Sigma'_1$ are Schubert varieties corresponding to $L$ and $L'$, respectively. If $X$ is a pure double line of genus $-1$, the locus in $\Gn$ of codimension two linear subspaces whose intersections with $X$ have length $\geq 2$ is a subvariety of $\Gn$ with cycle class $\Sigma^2_1$, which corresponds to the limit case when $\Sigma'_1$ approaches $\Sigma_1$. If $X$ consists of two incident lines or a double line contained in a plane $\Lambda$ with an embedded point at $p$, the corresponding locus is the subvariety $\Sigma_{1,1} \cup \Sigma_{2}$, where $\Sigma_{1,1}$ parameterizes codimension two linear subspaces intersecting $\Lambda$ and $\Sigma_{2}$ parameterizes those containing $p$. Therefore, $f$ is a well-defined surjective morphism. 

Note that $E'$ is the exceptional locus of $f$. A contracted fiber over a point in $f(E')$ parameterizes two incident lines with the same point-plane flag. So this fiber is isomorphic to $\mathbb P^{n-3}\times \mathbb P^{2}$, where $\mathbb P^{n-3}$ specifies the spatial embedded point pointing outward $\Lambda$ and $\mathbb P^{2}$ specifies the two lines passing through $p$ and contained in $\Lambda$. The curve classes $B_{2}$ and $B_{6}$ generate the cone of curves of this contracted fiber. 
Since $\psi_{F'}$ contracts the curve classes in the same way, it can be identified as $f$ and the model $P(F')$ is isomorphic to $\Theta_n$. 

Since the curve classes contracted by $\psi_{D_{1}}$ and $\psi_{D_{2}}$ are also contracted by $\psi_{F'}$, 
the morphism $\psi_{F'}$ factors through $P(D_{1})$ and $P(D_{2})$, respectively. Moreover, $\psi_{D_{1}}$ and $\psi_{D_{2}}$ both contract $E'$ and the image of $E'$ is of codimension $\geq 2$ in each target. So $P(D_{1})$ and $P(D_{2})$ are isomorphic in codimension one. By the formal definition of flips, $P(D_{1})\dashrightarrow P(D_{2})$ is a flip over $P(F')$ with respect to the divisor $F'$. 

(5) By $B_{4}\ldotp M'=B_{6}\ldotp M' = 0$, we know the morphism $\psi_{M'}$ contracts $N'$ and $E'$, since 
$B_{4}$ sweeps out $N'$ and $B_{6}$ sweeps out $E'$. 

Consider the Hilbert-Chow morphism from $W_{n}$ to Sym$^{2}\Go$ by sending a subscheme to its 
1-dimensional support with multiplicity. This map contracts $N'$ and $E'$ by forgetting double structures and embedded structures, respectively. 
An ample divisor on Sym$^{2}\Go$ can be defined as the locus of cycles intersecting a fixed codimension two linear subspace. Note that this divisor pulls back to $M'$ on $W_{n}$. Since Sym$^{2}\Go$ is normal, by Lemma~\ref{ample}, it is isomorphic to the model $P(M')$. 

(6) $W_{n}$ admits an $H_{3}$ fibration over $\mathbb G(3,n)$. By its definition, $R'$ is equivalent to the pull-back of $\sigma_{1}$ from
$\mathbb G(3,n)$. Hence, the morphism $\psi_{R'}$ contracts each fiber $H_{3}$ and the model $P(R')$ is isomorphic to $\mathbb G(3,n)$. 
\end{proof}

\begin{remark}
\label{quadrics}
There is another way to interpret the model $P(F')\cong\Theta_n$ in Proposition~\ref{Wn-models} (4). Consider the Hilbert scheme of quadric surfaces with class $\sigma_{n-1,n-3}+\sigma_{n-2,n-2}$ in the Pl\"{u}cker embedding of $\Go$. Lines parameterized by such a quadric $Q$ span a $\mathbb P^3$ in $\Pn$, which induces an inclusion $\mathbb G(1,3)\subset \Go$. 
Then $Q$ is uniquely determined by a codimension two linear section of $\mathbb G(1,3)$. Hence, this Hilbert scheme of quadrics is isomorphic to the model $P(D_2)\cong \Psi_n$ in Proposition~\ref{Wn-models} (2), which is a $\mathbb G(3,5)$ bundle over $\mathbb G(3,n)$. 

Associate to $Q$ the maximal subvariety $\Sigma$ in $\Gn$, where a linear subspace parameterized by $\Sigma$ contains some line parameterized by $Q$. Since $Q$ is a quadric surface in the Pl\"{u}cker embedding of $\Go$, $\Sigma$ is a codimension two linear section of the Pl\"{u}cker embedding of $\Gn$. In other words, this association maps $P(D_2)$ to the space $\Xi_n$ of codimension two linear sections of $\Gn$ such that the linear subspaces parameterized by
a section $\Sigma$ contain some line parameterized by the corresponding quadric $Q$ in $\Go$. We claim that the space $\Xi_n$ is isomorphic to $P(F')\cong\Theta_n$. 

When $Q$ is smooth, the lines parameterized by $Q$ all intersect two skew lines $L$ and
$L'$. The corresponding subvariety in $\Gn$ parameterizes linear subspaces 
that intersect both $L$ and $L'$. Hence, that codimension two linear section is the
intersection of two Schubert varieties $\Sigma_1\cap \Sigma'_1$. 

When $Q$ is singular but irreducible, by the proof of Theorem~\ref{3} for $n=3$, 
$Q$ parameterizes lines that intersect a fixed double line of genus $-1$ with length $\geq 2$. The corresponding codimension two linear section of $\Gn$ parameterizes
linear subspaces that intersect a fixed double line with length $\geq 2$.

When $Q$ is reducible, it is a union of two planes that are determined by a
flag $\{p \in \Lambda_2 \subset \Lambda_3\}$, where $p$ is a point, $\Lambda_2$ is a plane and $\Lambda_3$
is a $\mathbb P^3$. The two planes of $Q$ are determined by lines in $\Lambda_3$ passing through $p$ or contained in $\Lambda_2$, respectively. The locus of codimension two linear subspaces that contain
some line parameterized by $Q$ is reducible. It contains the Schubert variety $\Sigma_{1,1}$ of 
codimension two linear subspaces containing a line in $\Lambda_2$ and the Schubert variety $\Sigma_2$ parameterizing those linear subspaces containing $p$. In this case, no matter what $Q$ is, the image point in $\Xi_n$ only depends on the flag $p\in \Lambda_2$. 

Hence, we conclude that this association map $P(D_2)\rightarrow \Xi_n$ can be identified as $P(D_2)\rightarrow P(F')$, cf. the proof of Proposition~\ref{Wn-models} (4). 
\end{remark}

Let us calculate the canonical class of $W_n$. 

\begin{proposition}
\label{KW}
Let $n\geq 3$ be an integer. The canonical divisor $K_{W_n}$ has class $-(n+1)M' + (n-2)N' + (n-3)E'$. In particular, $W_n$ is a Fano variety. 
\end{proposition}

\begin{proof}
By Proposition~\ref{Wn-models} (1), we know $\pi$: $W_n\rightarrow \mbox{Bl}_\Delta\mbox{Sym}^2\Go\cong H_n$ contracts $E'$. 
The divisor $E'$ admits a $\mathbb P^{n-3}$ fibration over $\Gamma$, where $\Gamma$ parameterizes a pair of coplanar lines with a point $p$ at their intersection and the fiber specifies the embedded structure at $p$ pointing outward the plane. From this description, $\Gamma$ is a $\Pt$ fibration over the flag variety $\{p\in \Lambda_2\subset \Pn\}$, hence a smooth variety of dimension $3n-2$. Each fiber $\mathbb P^{n-3}$ of $E'$ gets contracted under $\pi$ to the base $\Gamma$. Therefore, we get 
$K_{W_n} = \pi^{*}K_{H_n} + (n-3)E' = -(n+1)M' + (n-2)N' + (n-3)E'.$ 

Since $N' = 2M'-2F'$ and $E' = 2F'-M'-R'$ on $W_n$, we have $-K_{W_n} = 2M'+2F'+(n-3)R'$, which is ample, cf. Proposition~\ref{Wn-chambers}. Hence, $W_n$ is a Fano variety. 
\end{proof}

Note that $H_n\cong$ Bl$_{\Delta}$Sym$^{2}\Gn\cong$ Bl$_{\Delta}$Sym$^{2}\Go$ appears as an intermediate model of $W_n$. Using the duality between $\Go$ and $\Gn$, the above results provide a recipe for analyzing the models induced by divisors on $H_n$. 

\begin{proof}[Proof of Theorem~\ref{3} for $n\geq 4$]
Part (1) is obvious because $(F, M)$ is the ample cone of $H_n$. 

For (2), since $N$ is the exceptional divisor of the blow-up and it is contained in the base locus of a divisor $D$ in $[M, N)$, after removing $N$, 
the model $P(D)$ is isomorphic to $P(M)$, which is the Chow variety Sym$^{2} \Gn$ parameterizing a pair of codimension two linear cycles in $\Pn$. 

Since $\Gn \cong \Go$, we can adapt the models obtained from $W_{n}$ to $H_n$.  
Note that $H_n$ is isomorphic to the model $P(D_{1})\cong \mbox{Bl}_\Delta\mbox{Sym}^2\Go$ in Proposition~\ref{Wn-models} (1). A pair of general codimension two linear subspaces corresponds to 
a pair of general lines. A double codimension two linear subspace corresponds to a double line. A pair of codimension two linear subspaces that span a hyperplane with an embedded component supported on their intersection corresponds to a pair of incident lines without specifying the embedded point, since the morphism $W_{n}\rightarrow P(D_{1})\cong H_n$ is induced by forgetting the spatial embedded structure of a subscheme. Via this translation, Theorem~\ref{3} (3) and (5) for $n\geq 4$ can be verified as follows. 

For (3), the morphism $H_n\rightarrow \Theta_n$ sends a subscheme $X$ to the locus of lines in $\Go$ that intersect $X$ with length $\geq 2$. This locus is a codimension two linear section 
of the Plucker embedding of $\Go$ as follows. If $X$ is a subscheme of type (I), the corresponding locus in $\Go$ is $\Sigma_1\cap \Sigma'_1$ parameterizing lines that intersect both components of $X$. For $X$ of type (II), the locus parameterizes lines whose intersections with $X$ contain double points. It is a subvariety 
of $\Go$ with cycle class $\Sigma_{1}^2$ corresponding to the limit case when $\Sigma'_1$ approaches $\Sigma_1$. For $X$ of type (III) or (IV), its Cohen-Macaulay part is contained in a hyperplane $\Lambda_{n-1}$ and its embedded component is supported on a subspace $\Lambda_{n-3}$. The corresponding locus in $\Go$ consists of two Schubert varieties $\Sigma_{1,1}\cup\Sigma_{2}$,  where $\Sigma_{1,1}$ parameterizes lines contained in $\Lambda_{n-1}$ and $\Sigma_{2}$ parameterizes lines intersecting $\Lambda_{n-3}$. 
Therefore, the morphism $H_n\rightarrow \Theta_n$ is well-defined. Furthermore, it is a small contraction by forgetting the components of a subscheme $X$ of type (III) or (IV) and only remembering the flag $(\Lambda_{n-3}\subset \Lambda_{n-1})$ determined by $X$. The contracted curve classes are the same as those contracted by the map $\psi_F$, hence the model
$P(F)$ is isomorphic to $\Theta_n$. 

For (5), we blow up $H_n$ along the locus (III) and the new space is isomorphic to $W_{n}$, which is an $H_{3}$ fibration over $\mathbb G(n-4, n)\cong\mathbb G(3,n)$. The blow-up corresponds to specifying a subspace $\mathbb P^{n-4}$ in the embedded $\mathbb P^{n-3}$ of a subscheme $X$ of type (III). By Proposition~\ref{Wn-models} (4), we can contract the exceptional divisor of the blow-up in a different way using the morphism $\psi_D$ induced by a divisor $D$ in the chamber $(E,F)$. The resulting model $P(D)\cong\Psi_n$ is a $\mathbb G(3,5)$ bundle over $\mathbb G(n-4, n)$, which is the desired flipping space. After the flip, the birational transform $E$ on $\Psi_n$ is equivalent to the pull-back of $\sigma_1$ from the base 
$\mathbb G(n-4,n)$. Hence, the map induced by $E$ contracts the $\mathbb G(3,5)$ bundle structure to the base $\mathbb G(n-4, n)\cong\mathbb G(3,n)$. 
\end{proof}

\end{document}